\documentclass[12pt,letterpaper]{amsart}
\usepackage{euler, epic,eepic,latexsym, amssymb, amscd, amsfonts, xypic, url, color, epsfig}
\input xy
\xyoption{all}

 \newlength{\baseunit}               % the basic unit length
 \newcount{\numlines}                % depth of picture (in number of lines)
\newtheorem{tm}{Theorem}

\newtheorem{lm}[tm]{Lemma}
\newtheorem{co}[tm]{Corollary}
\newtheorem{df}[tm]{Definition}

\newcommand{\hidden}[1]{\footnote{Hidden:  #1}}
\renewcommand{\hidden}[1]{}

\newcommand{\bbA}{\mathbf{A}}

\newcommand{\bbC}{\mathbf{C}}

\newcommand{\bbN}{\mathbf{N}}

\newcommand{\bbP}{\mathbf{P}}

\newcommand{\bbR}{\mathbf{R}}

\newcommand{\bbZ}{\mathbf{Z}}

\newcommand{\ShHom}{\mathscr{H}\kern -.5pt om}

\begin{document}
\pagestyle{plain}
\title{A classical proof that the algebraic homotopy class of a rational function is the  residue pairing }

\author{Jesse Leo Kass}

\address{Current: J.~L.~Kass, Dept.~of Mathematics, University of South Carolina, 1523 Greene Street, Columbia, SC 29208, United States of America}
\email{kassj@math.sc.edu}
\urladdr{http://people.math.sc.edu/kassj/}

\author{Kirsten Wickelgren}

\address{Current: K.~Wickelgren, School of Mathematics, Georgia Institute of Technology, 686 Cherry Street, Atlanta, GA 30332-0160}
\email{kwickelgren3@math.gatech.edu}
\urladdr{http://people.math.gatech.edu/~kwickelgren3/}


\date{\today}

\begin{abstract}
Cazanave has identified the  algebraic homotopy class of a rational function of $1$ variable with an explicit nondegenerate symmetric bilinear form.   Here we show that Hurwitz's proof of a classical result about real rational functions essentially gives an alternative proof of the stable part of  Cazanave's result. We also explain how this result can be interpreted in terms of the residue pairing and that this interpretation relates the result to the signature theorem of Eisenbud, Khimshiashvili, and Levine, showing that  Cazanave's result answers a question posed by Eisenbud for polynomial functions in $1$ variable.  Finally, we  announce results answering this question for functions in an arbitrary number of variables.
\end{abstract}
\maketitle

{\parskip=12pt % closing bracket is just before the bibliography blah

\section{Introduction}
In this paper we show that a classical proof of Hurwitz can be modified to give a new proof of  Cazanave's description of the degree of a rational function in $\bbA^1$-homotopy theory.   In \cite{hurwitz95}, Hurwitz computed the topological degree of a real rational function $f/g \in \operatorname{Frac} \bbR[x]$ where we assume  $f$ is monic and of degree $\mu>\deg(g)$.  The real rational function $f/g$ defines a self-map of the real projective line $\bbP^{1}_{\bbR}$, and the space $\bbP^{1}_{\bbR}(\bbR)$ of real points is homeomorphic to the $1$-sphere $S^1$, so the self-map  has a well-defined topological degree.  Hurwitz's theorem \cite[Section~3]{hurwitz95} is 

\begin{tm}[Hurwitz]  \label{Thm: HurwitzThm}
	The topological degree of $f/g \colon \bbP^{1}_{\bbR}(\bbR) \to \bbP^{1}_{\bbR}(\bbR)$ equals the signature of the real symmetric matrix 
\[
S(f/g)=
 \begin{pmatrix}
  s_{1} 		& s_{2} 		& \cdots 		&	s_{\mu-1}		& s_{\mu} \\
  s_{2} 		& s_{3} 		& \cdots 		&	s_{\mu}		& s_{\mu-1} \\
  \vdots  		& \vdots		& \ddots 		& \vdots			& \vdots   \\
%   s_{\mu-1}	& s_{\mu} & \cdots		&	s_{2 \mu - 3}		& s_{2 \mu - 2} \\
  s_{\mu} 	& s_{\mu+1} & \cdots 		&	s_{2 \mu-2}	& s_{2 \mu -1} 
 \end{pmatrix}
\]
for $s_1, s_2, \dots$ defined by $g(x)/f(x) = s_{1}/x + s_{2} /x^2 + s_{3}/x^3+\dots$.
\end{tm}

Cazanave proved an enrichment of Hurwitz's result in $\bbA^1$-homotopy theory.  The $\bbA^1$- or motivic degree of a rational function over an arbitrary field $k$ was constructed by Morel in celebrated work describing homotopy classes of self-maps of a motivic sphere.  The $\bbA^1$-degree of a rational function is valued in the Grothendieck--Witt group $\operatorname{GW}(k)$ of stable isomorphism classes of nondegenerate symmetric bilinear forms.  Morel's construction is indirect, but Cazanave proved the $\bbA^1$-degree has the following explicit description.

\begin{tm}[Cazanave] \label{Thm: Cazanave}
The $\bbA^1$-homotopical degree of $f/g \colon \bbP^{1}_{k} \to \bbP^{1}_{k}$ equals the class of 
\[
	\operatorname{B\acute{e}z}(f/g) :=\begin{pmatrix} 
								b_{1, 1}		&	\cdots	&	b_{1, \mu} \\
								\vdots		&	\ddots	&	\vdots	\\
								b_{\mu,1}	&	\cdots		&	b_{\mu, \mu}
											
	 \end{pmatrix} 
\]
for $b_{i,j}$ defined by 
\begin{equation} \label{Eqn: DefOfBezout}
	f(x) g(y) - f(y) g(x) = (x-y) \cdot \left( \sum_{i,j=1}^{\mu} b_{i, j} x^{i-1} y^{j-1} \right).
\end{equation}
\end{tm}
This is not explicitly stated in \cite{cazanave08, cazanave}, but it follows from \cite[Theorem~1.2, Corollary~3.10]{cazanave}, as we explain below. Cazanave also proved  results stronger than Theorem~\ref{Thm: Cazanave}, and we discuss them below as well.

The B\'{e}zout matrix $\operatorname{B\acute{e}z}(f/g)$ is conjugate to $S(f/g)$, so when $k=\bbR$, the real realization  (in the sense of \cite[Section~3.3]{morelvoevodsky1998}) of Cazanave's result is Hurwitz's result.  Cazanave's proof is, however, different from Hurwitz's proof, and one goal of this paper is to show that Hurwitz's proof can be modified to prove Theorem~\ref{Thm: Cazanave}.  

We present the modification of Hurwitz's proof in a way to make the result accessible to readers without a background in $\bbA^1$-homotopy theory.  In proving the result, we only make use of some basic properties of the degree in $\bbA^1$-homotopy theory that are straightforward analogues of familiar properties of the topological degree.  We recall the properties we use at the beginning of Section~\ref{Section: Classical Proof}, and a reader that knows no $\bbA^1$-homotopy theory but is willing to accept that these properties hold should be able to understand the proof.

To explain the idea of Hurwitz's proof and contrast it to Cazanve's, let us recall Cazanave's proof.  His proof makes fundamental use of ideas from homotopy theory.  There are two important sets of homotopy classes of self-maps of $\bbP^{1}_{k}$.  The first set is the set $[\bbP^{1}_{k}, \bbP^{1}_{k}]^{\bbA^1}$ of (stable pointed) $\bbA^{1}$-homotopy classes.  This set is constructed using model categories, and it naturally has the structure of a $\bbZ$-graded group.  Morel constructed the $\bbA^1$-degree as an isomorphism $\deg^{\bbA^1} \colon [\bbP^{1}_{k}, \bbP^{1}_{k}]^{\bbA^1} \to \operatorname{GW}(k)$ of graded groups.  

The second set of homotopy classes is the set $[\bbP^{1}_{k}, \bbP^{1}_{k}]^{\text{N}}$ of (pointed) naive homotopy classes.  This set is the quotient of the set of all self-maps of $\bbP^{1}_{k}$ by the equivalence relation defined by naively imitating the definition of homotopy in algebraic topology (with the unit interval replaced by the affine line).  Cazanave endows $[\bbP^{1}_{k}, \bbP^{1}_{k}]^{\text{N}}$ with the structure of a $\bbN$-graded monoid that makes the natural map $[\bbP^{1}_{k}, \bbP^{1}_{k}]^{\text{N}} \to [\bbP^{1}_{k}, \bbP^{1}_{k}]^{\bbA^1}$ into a monoid homomorphism. Theorem~\ref{Thm: Cazanave}  is then proven by showing that the rule sending $f/g$ to $\operatorname{B\acute{e}z}(f/g)$ defines a homomorphism $[\bbP^{1}_{k}, \bbP^{1}_{k}]^{\text{N}} \to \operatorname{GW}(k)$.  This proves the theorem as the rational functions $x/A$ with $A \in k^{\ast}$ generate  $[\bbP^{1}_{k}, \bbP^{1}_{k}]^{\text{N}}$, and for these functions, both  $\operatorname{B\acute{e}z}(f/g)$ and $\deg^{\bbA^1}(f/g)$ equal $\langle A \rangle$, the class of the bilinear form with Gram matrix  $\begin{pmatrix} A \end{pmatrix}$.

By contrast, Hurwitz's proof  uses matrix theory as follows.  The topological degree of $f/g$ equals the sum, over the real zeros of $f/g$, of the local topological degrees.  To identify this sum with the signature of $S(f/g)$, we pass to the complex numbers.  Over the complex numbers, $S(f/g)$ is conjugate to a block diagonal matrix $\operatorname{New}(f/g)$ with blocks indexed by the complex roots of $f(x)$.  The proof is  completed by computing the signature of $S(f/g)$ in terms of the blocks, computing that a pair of blocks corresponding to a complex conjugate pair of roots contributes $0$ to the signature and a block corresponding to a real root contributes the local degree at the root.  

We show that the same argument, suitably modified,  gives an independent proof that $\deg^{\bbA^1}(f/g)$ is the class of $S(f/g)$ or equivalently $\operatorname{B\acute{e}z}(f/g)$.  This argument, however, proves a result that is strictly weaker than the result proven by Cazanave.  The $\bbA^1$-degree of $f/g$ determines the \emph{stable} homotopy class of $f/g$, and  Cazanave proves that in fact $\operatorname{B\acute{e}z}(f/g)$ represents the \emph{unstable} homotopy class of $f/g$.  Recall  Morel proved that the unstable homotopy class of $f/g$ is determined by $\deg^{\bbA^1}(f/g)$ together with a scalar $d(f/g) \in k^{\ast}$ representing the class of the discriminant of $\deg^{\bbA^1}(f/g)$.  Cazanave proves that $d(f/g)$ equals the determinant of  $\operatorname{B\acute{e}z}(f/g)$.  This second result is not established by our modification of Hurwitz's proof, as we explain at the end of Section~\ref{Section: Classical Proof}.  We also do not study the monoid structure on $[\bbP^1_{k}, \bbP^{1}_{k}]^{\text{N}}$.  While the monoid structure plays a fundamental role in Cazanave's proof, it does not play a visible role in Hurwitz's proof. 

A secondary goal of this paper is to explain the relation of the work of Hurwitz and Cazanave to the beautiful signature theorem of  Eisenbud--Levine and Khimshiashvili.  The signature formula  identifies the local degree of a real polynomial function (possibly in many variables) as the signature of the residue pairing.  In Section~\ref{Section: Connection}, we explain  that the matrices $\operatorname{B\acute{e}z}(f/g)$ and $S(f/g)$ are Gram matrices for the residue pairing and that the signature formula generalizes   Hurwitz's theorem.  We announce the following theorem which generalizes the signature theorem and answers a question posed by Eisenbud in \cite{eisenbud78}:

\begin{tm}	[The main result  of \cite{wickelgren16}] \label{Thm: CompanionMainTheorem}
	The local $\bbA^1$-degree of $f \colon \bbA^{n}_{k} \to \bbA^{n}_{k}$ is the class  of the residue pairing $\beta_{0}(\operatorname{Res}_{0})$. 
\end{tm}
As we explain in Section~\ref{Section: Connection}, $\beta_{0}(\operatorname{Res}_{0})$ is a bilinear pairing on the localization of the algebra $Q_{0}(f/g) = k[x, g^{-1}]/(f g^{-1})$ at the ideal $(x)$.  The reader familiar with Cazanave's work may recall that he identifies $\operatorname{B\acute{e}z}(f/g)$ as a Gram matrix of a bilinear pairing under suitable hypotheses, but his bilinear pairing is not $\beta_{0}(\operatorname{Res}_{0})$. We explain the relation between the pairings in Section~\ref{Section: Connection}.

In this paper we will assume $\operatorname{char} k \ne 2$, primarily to simplify exposition.  Both Cazanave's theorem and Theorem~\ref{Thm: CompanionMainTheorem} remain valid in characteristic $2$, but this assumption simplifies the proof of  Theorem~\ref{Thm: NewCazanave}.  At one step in the proof, we reduce to the case where the preimage $(f/g)^{-1}(0)$ of the origin is \'{e}tale so that we can apply a formula (Equation~\eqref{Eqn: DegreeFormula} below) relating the global degree to the local degree, and the reduction makes use of the assumption $\operatorname{char} k \ne 2$ (specifically the consequence that a purely inseparable extension has odd degee). 

\section*{Conventions}

$k$ is a field of characteristic $\ne 2$.

A rational function is a nonconstant $k$-morphism  $f/g \colon \bbP^{1}_{k} \to \bbP^{1}_k$ satisfying $f/g(\infty)=\infty$. We always represent such a morphism by a fraction $f(x)/g(x) \in \operatorname{Frac} k[x]$ with $f$ monic, $f$ relatively prime to $g$, and $\deg(f) > \deg(g)$.  Observe that our rational functions are more properly called pointed rational functions, but we omit the term  ``pointed."

The Gram matrix of a symmetric bilinear form $\beta$ on a finite dimensional $k$-vector space $V$ with respect to a basis $e_1, e_2, \dots, e_{\mu}$ is the matrix $\begin{pmatrix} \beta(e_i, e_j) \end{pmatrix}$.

We write $\langle A \rangle \in \operatorname{GW}(k)$ for the class of the symmetric bilinear form with Gram matrix $\begin{pmatrix} A \end{pmatrix}$.

\section{A classical proof} \label{Section: Classical Proof}
Here we adapt Hurwitz's proof of Theorem~\ref{Thm: HurwitzThm} to prove Cazanave's result, Theorem~\ref{Thm: Cazanave}. The proof  is closely modeled on the Kre{\u\i}n--Na{\u\i}mark's exposition of Hurwitz's result in \cite{krein81}, and after giving the proof, we closely compare our argument with the argument in loc.~cit.  We then explain why Hurwitz's argument does not prove the stronger result that $\operatorname{B\acute{e}z}(f/g)$ represents the unstable homotopy class of $f/g$.

We begin by summarizing the basic properties of the $\bbA^1$-degree that we use.  The properties are all analogues of familiar properties of the usual  topological degree of a real rational function $f/g \colon \bbP^{1}_{\bbR} \to \bbP^{1}_{\bbR}$.  The reader unfamiliar with the degree map in $\bbA^1$-homotopy theory should be able to understand the proof of Theorem~\ref{Thm: NewCazanave} using just the properties we summarize below (up to Lemma~\ref{Lemma: Local Invariance}).

The degree map in $\bbA^1$-homotopy theory takes values in the the Grothendieck--Witt $\operatorname{GW}(k)$.  This is the group completion of the monoid of nondegenerate quadratic forms.  The theory supports both a global degree $\deg^{\bbA^1}(f/g)$ of a rational function $f/g$ and a local degree $\deg^{\bbA^1}_{r}(f/g)$ at a closed point $r \in \bbP^1_{k}$ with residue field $k$.  These degrees satisfy the following properties:
\begin{itemize}
	\item  	the degree of a composition is the product of the degrees, 
	\item 	the global degree can be expressed as a sum of traces of local degrees, 
	\item 	both the local and global degrees are invariant under naive $\bbA^1$-homotopies, 
	\item		the global degree is normalized so that $\deg^{\bbA^1}(x/A) =  \langle A \rangle$.
\end{itemize}
The expression of the global degree as a sum of local degrees is the following one.  Suppose $f/g \colon \bbP_{k}^1 \to \bbP_{k}^1$ is a rational function with the property that preimage of the origin is a union $\{ r_1, \dots, r_m \}$ of closed points with residue fields that are separable extensions of $k$.  Then the global $\bbA^1$-degree satisfies
\begin{equation} \label{Eqn: DegreeFormula}
		\deg^{\bbA^1}(f/g) = \operatorname{Tr}_{k[r_1]/k}( \deg_{r_1}^{\bbA^1}( f/g \otimes k[r_1] )) + \dots +\operatorname{Tr}_{k[r_m]/k}( \deg_{r_m}^{\bbA^1}( f/g \otimes k[r_m] )).
\end{equation}
Here $\operatorname{Tr} \colon \operatorname{GW}(L) \to \operatorname{GW}(k)$ is the trace function, which is defined by sending a bilinear pairing $ \beta \colon V \times V \to L$ to the composition $\operatorname{Tr}_{L/k} \circ \beta$ of $\beta$ with the field trace.   Under additional hypotheses, the formula is described in \cite[Section~2]{morel04}, \cite[page~1036]{morel06}, \cite[Section~5.2]{LevineMilan}.  The formula can be deduced from \cite[Proposition~13]{wickelgren16} and \cite[Proposition~32]{wickelgren17}.

The invariance property of the global degree we use is that 
\[
	\deg^{\bbA^1}(f_0/g_0) = \deg^{\bbA^1}(f_1/g_1)
\]
when there exists a morphism $H \colon \bbP^{1}_{k} \times_{k} \bbA^{1}_{k} \to \bbP^{1}_{k}$ such that $f_0/g_0(x) = H(x, 0)$, $f_1/g_1(x) = H(x,1)$, and $H(\infty,t) = \infty$ (i.e.~when there exist a naive $\bbA^1$-homotopy).  The local analogue is the following lemma.

\begin{lm} \label{Lemma: Local Invariance}
	Let $f_0/g_0, f_1/g_1 \colon \bbP^1_{k} \to \bbP^{1}_{k}$ be rational functions satisfying $f_{0}(0) = f_{1}(0)=0$.  If there exists a morphism $H \colon \bbP^{1}_{k} \times_{k} \bbA_{k}^{1} \to \bbP_{k}^{1}$ such that 
	\begin{gather*}
		H(x, 0)= f_{0}/g_{0}(x), H(x, 1) = f_{1}/g_{1}(x), \text{ and} \\
		\{ 0 \} \times_{k} \bbA^{1}_{k} \text{ is a connected component of $H^{-1}( \{ 0 \} \times \bbA^{1}_{k})$,}
	\end{gather*}
	then 
	\[
		\deg^{\bbA^1}_{0}(f_0/g_0) = \deg_{0}^{\bbA^1}(f_1/g_1).
	\]
\end{lm}
\begin{proof}
	The local degree $\deg^{\bbA^1}_{0}(f/g)$ of a rational function $f/g$ with an isolated zero at the origin is defined using the purity theorem \cite[Section~3, Lemma 2.1]{morelvoevodsky1998}.  If $U \subset \bbP^{1}_{k}$ is the complement of the zeros distinct from the origin, then $f/g \colon \bbP^{1}_{k} \to \bbP^{1}_{k}$ induces a morphism of quotient spaces
	\begin{equation} \label{Eqn: MapOnQuotients}
		\frac{U}{U-\{0\}} \to \frac{\bbP^{1}_{k}}{\bbP^{1}_{k}-\{0\}}.
	\end{equation}
	The purity theorem identifies both the target and source quotient space with $\bbP^{1}_{k}$, so the morphism \eqref{Eqn: MapOnQuotients} has a well-defined global $\bbA^{1}$-degree that is defined to be $\deg^{\bbA^1}_{0}(f/g)$.
	
	Now suppose that we are given $H$ as in the statement of the lemma.  We use $H$ to construct a global naive $\bbA^{1}$-homotopy as follows.   Let $Z$ be the union of the connected components of $H^{-1}( \{ 0 \} \times_{k} \bbA^{1}_{k})$ that are distinct from $\{ 0 \} \times_{k} \bbA^{1}_{k}$.  We have the identification
\begin{align*}
	\frac{\bbP^{1}_{k} \times_{k} \bbA^{1}_{k}}{\bbP^{1}_{k} \times_{k} \bbA^{1}_{k} - H^{-1}(0)}	=& 	\frac{\bbP^{1}_{k} \times_{k} \bbA^{1}_{k}}{\bbP^{1}_{k} \times_{k} \bbA^{1}_{k} - (\{x=0\} \cup Z)}	\\
		=& \frac{\bbP^{1}_{k} \times_{k} \bbA^{1}_{k}}{\bbP^{1}_{k} \times_{k} \bbA^{1}_{k}-\{ x=0\}} \vee \frac{\bbP^{1}_{k} \times_{k} \bbA^{1}_{k}}{\bbP^{1}_{k} \times_{k} \bbA^{1}_{k}-Z}
\end{align*}
Thus the morphism $\frac{\bbP^{1}_{k} \times_{k} \bbA^{1}_{k}}{\bbP^{1}_{k} \times_{k} \bbA^{1}_{k} - H^{-1}(0)} \to \frac{\bbP^{1}_{k}}{\bbP^{1}_{k}-0}$ induced by $H$ can be identified with a morphism $\frac{\bbP^{1}_{k} \times_{k} \bbA^{1}_{k}}{\bbP^{1}_{k} \times_{k} \bbA^{1}_{k}-\{ x=0\}} \vee \frac{\bbP^{1}_{k} \times_{k} \bbA^{1}_{k}}{\bbP^{1}_{k} \times_{k} \bbA^{1}_{k}-Z} \to \frac{\bbP^{1}_{k}}{\bbP^{1}_{k}-0}$.  Pre-composing this morphism with the natural morphism $\frac{\bbP^{1}_{k}}{\bbP^{1}_{k} - \{ 0 \}} \times_{k} \bbA^{1}_{k} \to \frac{\bbP^{1}_{k} \times_{k} \bbA^{1}_{k}}{\bbP^{1}_{k} \times_{k} \bbA^{1}_{k}-\{ x=0\}} \to \frac{\bbP^{1}_{k} \times_{k} \bbA^{1}_{k}}{\bbP^{1}_{k} \times_{k} \bbA^{1}_{k}-\{ x=0\}} \vee \frac{\bbP^{1}_{k} \times_{k} \bbA^{1}_{k}}{\bbP^{1}_{k} \times_{k} \bbA^{1}_{k}-Z}$, we obtain a naive $\bbA^1$-homotopy from the morphism induced by $f_0/g_0$ to the morphism induced by $f_1/g_1$.
\end{proof}

We first use these properties to compute the degree of a power map.
\begin{lm} \label{Lemma: LemmaComputePower}
	For $A \in k^{\ast}$, we have
	\begin{equation} \label{Eqn: FormulaForPower}
		\deg^{\bbA^{1}}(x^{\mu}/A) = \begin{cases} \langle A \rangle +  \frac{\mu-1}{2} \cdot \langle 1, -1 \rangle & \text{ $\mu$ odd;}\\
																	\frac{\mu}{2} \cdot \langle 1, -1 \rangle & \text{ $\mu$ even.} \end{cases}
	\end{equation}
\end{lm}
\begin{proof}
	First, we prove the case where $A=1$ by  induction on $\mu$.  The result holds for $\mu=1$ by the construction of $\deg^{\bbA^{1}}$, so we assume the result holds for $\mu$ and then prove  it holds for $\mu+1$.  Consider the auxiliary function $x^{\mu+1}+x^{\mu} \colon \bbP^{1}_{k} \to \bbP^{1}_{k}$.  We compute the degree in two different ways.

First, the expression $x^{\mu+1} + t x^{\mu}$ defines a naive $\bbA^1$-homotopy from $x^{\mu+1} + x^{\mu}$ to $x^{\mu+1}$, so
\begin{equation}
	\deg^{\bbA^{1}}(x^{\mu+1} +  x^{\mu})  = \deg^{\bbA^{1}}(x^{\mu+1}).
\end{equation}

Second, by Formula~\eqref{Eqn: DegreeFormula}, we have that $\deg^{\bbA^1}(x^{\mu+1} +  x^{\mu}) = \deg^{\bbA^{1}}_{0}(x^{\mu+1} +  x^{\mu})+\deg^{\bbA^{1}}_{-1}(x^{\mu+1} + x^{\mu})$.  To compute the local degrees, observe that $(x,t) \mapsto x^{\mu}(1+t x)$
defines a naive local $\bbA^1$-homotopy in the sense of Lemma~\ref{Lemma: Local Invariance}, so 
\[
	\deg_{0}^{\bbA^1}(x^{\mu+1} +  x^{\mu}) = \deg^{\bbA^1}_{0}(x^\mu).
\]
A similar argument shows
\[
	\deg_{1}^{\bbA^1}(x^{\mu+1} +  x^{\mu}) =\deg^{\bbA^{1}}_{-1} ((-1)^{\mu} (x+1)).
\] 
We now compute
\begin{align*}
	\deg^{\bbA^{1}}(x^{\mu+1})	&= 	\deg^{\bbA^{1}} (x^{\mu+1} +  x^{\mu}) \\
							& = \deg^{\bbA^{1}}_0 (x^{\mu}) + \deg^{\bbA^{1}}_{-1} ((-1)^{\mu} (x+1)) \\
							& = \deg^{\bbA^{1}} (x^{\mu}) + \langle(-1)^{\mu}\rangle,
\end{align*}
completing  induction. 

We complete the proof by observing that $\deg^{\bbA^1}(x^{\mu}/A) = \deg^{\bbA^1}(x/A) \cdot \deg^{\bbA^1}(x^\mu)$ because $\deg^{\bbA^1}$ transforms composition into multiplication.

\end{proof}

The next lemma identifies the element of $\operatorname{GW}(k)$ from the last lemma with an explicit matrix.

\begin{lm} \label{Lm: ComputeNewtonClass}
	If $A(1), A(2), \dots, A(\mu) \in k$ are scalars with $A(\mu)$ nonzero, then the nondegenerate symmetric matrix
	\[
		 \begin{pmatrix}
			A(1) 				& A(2) 		& \cdots 			&	A(\mu-1)	& A(\mu) \\
			A( 2) 			& A(3)		 & \cdots 			&	A(\mu)	& 0 \\
			\vdots  			& \vdots 	 	& \ddots 			& \vdots		& \vdots   \\
			A(\mu-1)			& A(\mu) 	& \cdots			&	0		& 0 \\
			A(\mu) 			& 0 			& \cdots			&	0		& 0
 \end{pmatrix}
	\]
	represents the Grothendieck--Witt class
	\[
		w = \begin{cases} \langle A(\mu) \rangle +  \frac{\mu-1}{2} \cdot \langle 1, -1 \rangle & \text{ $\mu$ odd;}\\
																	\frac{\mu}{2} \cdot \langle 1, -1 \rangle & \text{ $\mu$ even.} \end{cases}
	\]
\end{lm}
\begin{proof}
	Let $x_1, \dots, x_{\mu}$ denote the basis dual to the standard basis on $k^{\oplus \mu}$, so the bilinear form $\beta$ defined by the given matrix is $\sum A(i+j-1) x_i \otimes x_j$. Rewrite the bilinear form as
	\begin{gather} \label{Eqn: NewtonClassBasis}
			x_{1} \otimes \Psi_{1}+ \Psi_{1} \otimes x_{1}  + x_{2} \otimes \psi_{2} +\psi_{2} \otimes x_{2} + \dots x_{\mu/2} \otimes  \Psi_{\mu/2}+ \Psi_{\mu/2} \otimes x_{\mu/2} 		\text{ if $\mu$ even;} \\
			x_{1} \otimes \Psi_{1}+ \Psi_{1} \otimes x_{1}  + \dots x_{(\mu-1)/2} \otimes  \Psi_{(\mu-1)/2}+ \Psi_{(\mu-1)/2} \otimes x_{(\mu-1)/2} + A(\mu) x_{(\mu+1)/2} \otimes x_{(\mu-3)/2}  		\text{ if $\mu$ odd}  \notag
	\end{gather}
	for $\Psi_1, \Psi_2, \dots$ the linear functions
	\begin{align*}
		\Psi_1 =& \frac{A(1)}{2} x_1 + A(2) x_2 +A(3) x_{3} + \dots + A(\mu-2) x_{\mu-2}+A(\mu-1) x_{\mu-1}+A(\mu) x_{\mu}, \\
		\Psi_2 =& \phantom{A(1) x_1 + }\frac{A(3)}{2} x_2 +A(4) x_{3} + \dots + A(\mu-1) x_{\mu-2}+A(\mu) x_{\mu-1},\phantom{+A(\mu) x_{\mu},} \\
		\Psi_3 =& \phantom{A(1) x_1 + A(2) x_2 +}\frac{A(5)}{2} x_{3} + \dots + A(\mu) x_{\mu-2}, \phantom{+A(\mu-1) x_{\mu-1}+A(\mu) x_{\mu},} \\
		\dots.
	\end{align*}
	The elements $x_1, \Psi_1, x_2, \Psi_2, \dots$ form a dual basis, so \eqref{Eqn: NewtonClassBasis} shows that  $\beta$ is the orthogonal sum of the  hyperbolic planes $x_{1} \otimes \Psi_{1}+ \Psi_{1} \otimes x_{1}, x_{2} \otimes \Psi_{2}+ \Psi_{2} \otimes x_{2}, \dots$ and, when $\mu$ is odd, the $1$-dimensional space $A(\mu) x_{(\mu-1)/2} \otimes x_{(\mu-1)/2}$.
\end{proof}

We now introduce an auxiliary matrix, the Newton matrix $\operatorname{New}(f/g)$, which will be used in the proof.
\begin{df}

	Given a root $r \in k$ of  $f$ of multiplicity $\mu_0$, let 
\begin{equation} \label{Eqn: PartialFractions}
	g(x)/f(x) = \frac{A_{r}(\mu_{0})}{(x-r)^{\mu_{0}}}+\frac{A_{r}(\mu_{0}-1)}{(x-r)^{\mu_{0}-1}}+\dots+ \frac{A_{r}(2)}{(x-r)^{2}}+\frac{A_{r}(1)}{(x-r)}+\text{higher order terms}
\end{equation}
be a partial fractions decomposition.  Define the \textbf{local Newton matrix}
\begin{equation} \label{Eqn: NewtonBlock}
	\operatorname{New}(f/g, r) :=  \begin{pmatrix}
  A_{r}(1) 				& A_{r}(2) 				& \cdots 			&	A_{r}(\mu_{0}-1)		& A_{r}(\mu_{0}) \\
  A_{r}( 2) 				& A_{r}(3)				& \cdots 			&	A_{r}(\mu_{0})			& 0 \\
  \vdots  				& \vdots  				& \ddots 			& \vdots  					& \vdots \\
  A_{r}(\mu_{0}-1)		& A_{r}(\mu_{0}) 		& \cdots			&	0					& 0 \\
  A_{r}(\mu_{0}) 		& 0 					& \cdots			& 	0					& 0
 \end{pmatrix}.
\end{equation}
If $k$ contains all the roots $\{ r_1, \dots, r_m\}$ of $f(x)$, then we define the \textbf{Newton matrix}
\[
	\operatorname{New}(f/g) :=  \begin{pmatrix}
	\operatorname{New}(f/g, r_1) 		& 0 						& \cdots 	&	0		& 0\\
	0 							& \operatorname{New}(f/g, r_2) 	& \cdots 	&	0		& 0 \\
	\vdots  						& \vdots 					& \ddots 	& \vdots		& \vdots   \\
	0							& 0 & \cdots						&	\operatorname{New}(f/g, r_{m-1})		& 0 \\
	0 							& 0 & \cdots 						&	0	& \operatorname{New}(f/g, r_m) 
 \end{pmatrix}.
\]
\end{df}

The previous two lemmas shows that $\deg^{\bbA^1}(f/g)$ equals the class of $\operatorname{New}(f/g)$ when $f/g = x^{\mu}/A$.  The following lemma establishes equality in greater generality.

\begin{co} \label{Corollary: NewtonIsDegreeLocal}
	If $r \in k$ is a root of $f$, then $\operatorname{New}(f/g,r)$ represents $\deg_{r}^{\bbA^1}(f/g)$. 
\end{co}
\begin{proof}
	Say that the multiplicity of $r \in k$ in $f$ is $\mu_0$ and $A$ equals $g(0)$.  Write  $f(x) = x^{\mu_{0}} f_0(x)$ with $f_{0}(0) \ne 0$ and $g(x) = x g_{0}(x)+A$.  The expression 
	\[
		H(x, t) = (x^{\mu}+t f_{0}(x))/( t x g_{0}(x) + A)
	\]
	defines a naive local $\bbA^1$-homotopy (in the sense of Lemma~\ref{Lemma: Local Invariance}), showing $\deg^{\bbA^1}_{0}(f/g) = \deg^{\bbA^1}_{0}(x^{\mu_0}/A)$.  We thus have reduced the claim to the special case $f/g = x^{\mu_0}/A$,  which is  Lemmas~\ref{Lemma: LemmaComputePower}  and \ref{Lm: ComputeNewtonClass}.
\end{proof}

\begin{co} \label{Corollary: NewtonIsDegree}
	If $k$ contains all the roots of $f(x)$, then $\operatorname{New}(f/g)$ represents the class of $\deg^{\bbA^1}(f/g)$. 
\end{co}
\begin{proof}
	This is immediate from Corollary~\ref{Corollary: NewtonIsDegreeLocal} because $\deg^{\bbA^1}(f/g)$ is the sum of the local degrees $\deg^{\bbA^1}_{r_i}(f/g)$ by Formula~\eqref{Eqn: DegreeFormula} and $\operatorname{New}(f/g)$ is the sum of the terms $\operatorname{New}(f/g,r_i)$ by definition.
\end{proof}

\begin{tm} \label{Thm: NewCazanave}
	The matrix $S(f/g)$ represents  $\deg^{\bbA^1}(f/g) \in \operatorname{GW}(k)$.
\end{tm}
\begin{proof}
	First, assume $k$ contains all  the roots of $f(x)$, so $\operatorname{New}(f/g)$ is defined over $k$.  We will show that $S(f/g)$ is conjugate to $\operatorname{New}(f/g)$ by manipulating the bilinear polynomial $\sum s_{i+j-1} a_i b_j$ defined by $S(f/g)$.  Consider the expression 
	\begin{equation} \label{Eqn: RationalTimesTheta}
		g(x)/f(x) \cdot \Theta(a; x) \Theta(b; x)
	\end{equation}
		 for $\Theta$ the polynomial 
	\[
		\Theta(a; x) = a_0 + a_1 x + a_2 x^2 + \dots + a_{\mu-1} x^{\mu-1}.
	\]
	The coefficient of $1/x$ in \eqref{Eqn: RationalTimesTheta} is $\sum s_{i+j+1} a_i b_j$, as we see by substituting $s_{1}/x+s_{2}/x^2 + \dots$ for $g(x)/f(x)$ and multiplying.
	
	Alternatively the coefficient of $1/x$ is the sum of the residues at the roots of $f$.  Here we take the residue at a root $r$ algebraically as the coefficient of $1/(x-r)$ in a series expansion. If $r$ is a root of $f(x)$ with multiplicity $\mu_0$, write
	\begin{align*}
		\Theta(a; x) =& a_0 +a_1 (x-r+r)+a_2 (x-r+r)^2+ \dots a_{\mu-1} (x-r+r)^{\mu-1} \\
				=& (a_0+a_1 r+a_2 r^2+\dots+a_{\mu-1} r^{\mu-1}) + (a_1+2 r a_2+\dots) (x-r)+\dots + a_{\mu-1} (x-r)^{\mu-1} \\
				=& \Theta_{0}(a;r)+ \Theta_{1}(a;r) (x-r) + \dots+ \Theta_{\mu-1}(a;r) (x-r)^{\mu-1}
	\end{align*}
	for some explicit polynomials $\Theta_0, \Theta_1, \dots, \Theta_{\mu-1}$ that are linear in the $a_{i}$'s. Substituting these expressions and $g(x)/f(x) = \frac{A_{r}(\mu_{0})}{(x-r)^{\mu_{0}}}+\frac{A_{r}(\mu_{0}-1)}{(x-r)^{\mu_{0}-1}}+\dots$ into \eqref{Eqn: RationalTimesTheta}, we get that the residue at $x=r$ is 
	\[
		\sum A_{r}(i+j+1) \Theta_{i}(a; r) \Theta_{j}(b; r).
	\]
	Summing over all roots, we get that the coefficient of $1/x$ is 
	\[	
		\sum A_{r_{1}}(i+j+1) \Theta_{i}(a; r_1) \Theta_{j}(b; r_1) + 	\dots+	\sum A_{r_{m}}(i+j+1) \Theta_{i}(a; r_m) \Theta_{j}(b; r_m).
	\]
	As a bilinear polynomial in the $\Theta_i$'s, this is the bilinear polynomial associated to $\operatorname{New}(f/g)$.  Thus if $M$ is the change-of-basis matrix relating the $a_i$'s to the $\Theta_i(a; r_j)$'s, then 
	\begin{equation} \label{Eqn: StoNewton}
		S(f/g) = M \cdot \operatorname{New}(f/g) \cdot M^{\operatorname{T}}.
	\end{equation}
	The matrix $M$ is invertible ($M$ is the confluence Vandermonde matrix), so we conclude from Corollary~\ref{Corollary: NewtonIsDegree} that $S(f/g)$ represents $\deg^{\bbA^1}(f/g)$.

	When $f(x)$ is arbitrary, we extend the above argument using the theory of descent (as in e.g.~\cite[Chapter~III]{knus}) as follows.  Fix a splitting field $L$ for $f(x)$.  We can assume that $L/k$ is Galois because the natural map from $\operatorname{GW}(k)$ to the Grothendieck--Witt group of the separable closure is injective (recall $\operatorname{char} k \ne 2$). Denote the irreducible factors of $f(x)$ by $\pi_{1}(x), \dots, \pi_{m}(x)$.  For each factor $\pi(x)$, pick an embedding of the residue field $k[x]/\pi(x)$ into $L$ and let $r_i \in L$ denote the image of $x$.  By Formula~\eqref{Eqn: DegreeFormula} and Corollary~\ref{Corollary: NewtonIsDegreeLocal}, we have
	\begin{equation}  \label{Eqn: SumOfTraces}
		\deg^{\bbA^1}(f/g) = \operatorname{Tr}_{k[r_1]/k}( \operatorname{New}( f/g \otimes k[r_1], r_1)) + \dots + \operatorname{Tr}_{k[r_m]/k}( \operatorname{New}( f/g \otimes k[r_m], r_m)). 
	\end{equation}

	To show that  \eqref{Eqn: SumOfTraces} and $S(f/g)$ define isomorphic bilinear spaces, it is enough to extend scalars to $L$ and construct an isomorphism over $L$ that respects the natural descent data.  The space obtained from \eqref{Eqn: SumOfTraces}  is the  space defined by 
	\begin{gather*}
		\left(\oplus_{\sigma \in \operatorname{Gal}(L/k[r_1])} \operatorname{New}(f/g \otimes L, \sigma(r_1))\right) \oplus \dots \oplus \left(\oplus_{\sigma \in \operatorname{Gal}(L/k[r_m])} \operatorname{New}(f/g \otimes L, \sigma(r_m)) \right)	\\
		=  \operatorname{New}(f/g \otimes L).
	\end{gather*}
	by \cite[Theorem~6.1]{lam05}. The descent data on this space is the descent data associated to the permutation action of $\operatorname{Gal}(L/k)$ on the roots of $f(x)$, and the isomorphism defined by the matrix $M$ from \eqref{Eqn: StoNewton} transforms this descent data into the tautological descent data on the bilinear space defined by $S(f/g)$.
\end{proof}

\begin{co} \label{Corollary: BezoutToS}
	The matrix $\operatorname{B\acute{e}z}(f/g)$ represents $\deg^{\bbA^1}(f/g)$.
\end{co}
\begin{proof}
	It is enough to show that $\operatorname{B\acute{e}z}(f/g)$ is congruent to $S(f/g)$.  Writing  $f(x) = x^{\mu} + a_{1} x^{\mu-1}+\dots+a_{\mu}$, we compute
	\begin{align}
		\frac{f(x) g(y) - f(y) g(x)}{x-y}	=& f(x) f(y) \frac{\frac{g(y)}{f(y)} - \frac{g(x)}{f(x)}}{x-y}	\notag \\
								=& f(x) f(y)  \frac{\sum_{i=1}^{\infty} s_{i} (y^{-i}-x^{-i})}{x-y}				\notag \\
								=& f(x) f(y)  \frac{\sum_{i=1}^{\infty} s_{i} (x^{i}-y^{i}) x^{-i} y^{-i}}{x-y}				\notag \\
								=& f(x) f(y) \sum^{\infty}_{i,j=1} s_{i+j-1} x^{-i} y^{-j}				\notag  \\
								=& \sum_{i, j=1}^{\infty} s_{i+j-1} ( x^{\mu-i} + a_{1} x^{\mu-i-1} + \dots + a_{\mu} x^{-i}) ( y^{\mu-j} + a_{1} y^{\mu-j-1} + \dots + a_{\mu} y^{-j}).\label{Eqn: BezoutExpression}
	\end{align}
 The above expression equals $\sum b_{i, j} x^{i-1} y^{j-1}$, so the terms in  \eqref{Eqn: BezoutExpression} where $x$ or $y$ appears with negative degree must cancel out, and we deduce
	\[
		\sum_{i,j=1}^{\mu} b_{i, j} x^{i-1} y^{j-1} = \sum_{i, j=1}^{\mu} s_{i+j-1} ( x^{\mu-i} + a_{1} x^{\mu-i-1} + \dots + a_{\mu-i} ) ( y^{\mu-j} + a_{1} y^{\mu-j-1} + \dots + a_{\mu-j}).
	\]
	This shows that  the relevant change-of-basis matrix conjugates $\operatorname{B\acute{e}z}(f/g)$ to $S(f/g)$.
\end{proof}

Let us compare the argument just given with  Kre{\u\i}n--Na{\u\i}mark's treatment of Hurwitz's result in \cite{krein81}.  We deduced the result that $\operatorname{B\acute{e}z}(f/g)$ represents $\deg^{\bbA^1}(f/g)$ (Corollary~\ref{Corollary: BezoutToS}) from the analogous result about $S(f/g)$ (Theorem~\ref{Thm: NewCazanave}) in the same way that Kre{\u\i}n--Na{\u\i}mark do on \cite[page~277]{krein81}.  Our proof of the result about $S(f/g)$ is similar to  the argument on \cite[pages~280--282]{krein81} but with two significant differences.  First, both we and Kre{\u\i}n--Na{\u\i}mark compute the class of $S(f/g)$ by passing to a splitting field $L$ for $f(x)$ and working with $\operatorname{New}(f/g \otimes L)$, but the details are different.  Kre{\u\i}n--Na{\u\i}mark only consider $k=\bbR$, and they compute the signature of $S(f/g)$ from $\operatorname{New}(f/g \otimes L)$ by observing that a complex conjugate pair of roots correspond to a summand of $S(f/g)$ with signature zero.  We allow $k$ to be arbitrary (with $\operatorname{char} k \ne 2$), and in contrast to the case $k=\bbR$, a Galois orbit of roots can contribute to the class of $S(f/g)$ in a complicated way, so instead of directly computing  we use descent theory.

The second difference occurs in our proof of Corollary~\ref{Corollary: NewtonIsDegreeLocal}, which identifies the local degree $\deg^{\bbA^1}_{r}(f/g)$ with the class of $\operatorname{New}(f/g, r)$.  The corollary is deduced from Lemmas~\ref{Lemma: LemmaComputePower} and \ref{Lm: ComputeNewtonClass}.  Lemma~\ref{Lm: ComputeNewtonClass} is the computation of  the class of  $\operatorname{New}(f/g, r)$, and the computation is the same as the one appearing on \cite[page~281]{krein81}.  Lemma~\ref{Lemma: LemmaComputePower}, however, does not have a direct analogue in \cite{krein81}.  That lemma computes the local degree as the  expression  in Lemma~\ref{Lm: ComputeNewtonClass}, and for Kre{\u\i}n--Na{\u\i}mark, the analogous fact does not need a proof since it is their definition of the local degree (or rather local Cauchy index; see Section~\ref{Section: References} below).

The proof just given  provides a second proof of Cazanave's result that $\operatorname{B\acute{e}z}(f/g)$ represents $\deg^{\bbA^1}(f/g)$.  Cazanave in fact proves the stronger result that  $\operatorname{B\acute{e}z}(f/g)$ represents the unstable homotopy class of $f/g$, and our proof does not yield this stronger result.  The reason for this is clearly demonstrated by the special case where $f(x)$ has $\mu$ distinct roots defined over $k$.  We deduced the result about $\operatorname{B\acute{e}z}(f/g)$ from the analogous result about $\operatorname{New}(f/g)$, and $\operatorname{New}(f/g)$ does not represent the unstable homotopy class of $f/g$.  Indeed, the unstable homotopy class is determined by  $\deg^{\bbA^1}(f/g)$ and a scalar $d(f/g) \in k^{\ast}$ that represents the discriminant of $\deg^{\bbA^{1}}(f/g)$.  The Newton matrix is a diagonal matrix with diagonal entires $g(r_1)/f'(r_1), \dots, g(r_\mu) /f'(r_\mu)$, so 
\begin{align}
	\det( \operatorname{New}(f/g)) =& g(r_1)/f'(r_1) \cdots g(r_{\mu})/f'(r_\mu) \text{ but }  \label{Eqn: FormulaForDet} \\
	\det( \operatorname{B\acute{e}z}(f/g)) =& (-1)^{\mu(\mu-1)/2} \operatorname{Resultant}(f,g) \notag \\
									=& \operatorname{Disc}(f) \cdot g(r_1)/f'(r_1) \cdots g(r_{\mu})/f'(r_{\mu}).  \notag
\end{align}
Here $\operatorname{Resultant}$ is the resultant, and $\operatorname{Disc}$ the discriminant.

Equation~\eqref{Eqn: FormulaForDet} shows that, unlike $\deg^{\bbA^1}(f/g)$,  $d(f/g)$ is not determined by the derivatives of $f(x)/g(x)$ at the roots.  For example, both $f_1(x)/g_{1}(x) = x^2-x$  and $f_{2}(x)/g_{2}(x) = (x^2-1)/2$ have two simple zeros at which the values of the derivative are $+1$ and $-1$ respectively, but the functions are not unstably homotopic since $d(f_1/g_1) = -1$ but   $d(f_2/g_2)=-4$. This should not be surprising as the formula relating  the global degree to a sum of local degrees is proven using stable techniques.

\subsection{References} \label{Section: References}
The authors used Kre{\u\i}n--Na{\u\i}mark's survey \cite{krein81} as the primarily source on Hurwitz's theorem.  The proof in loc.~cit.~is essentially the same as Hurwitz's proof in  \cite[Section~3]{hurwitz95}.  Other treatments of the result include \cite[page~210, Theorem~9]{gantmacher64}, \cite[Theorem~10.6.5]{rahman02}, \cite[Theorem~9.4]{basu06}, and \cite[Theorem~8.59]{fuhrmann12}.  All of these references do not discuss the topological degree and instead discuss the Cauchy index, an equivalent invariant.  The local Cauchy index $\operatorname{ind}_{r}(f/g)$ at a pole $r \in \bbR$ of $f/g$ (i.e.~at a root of $g$) is
\[
	\operatorname{ind}_{r}(f/g) := \begin{cases}
		+1 & \text{if $f/g$ jumps from $-\infty$ to $+\infty$ at $r$;}\\
		-1 & \text{if $f/g$ jumps from $+\infty$ to $-\infty$ at $r$;}\\
		0	& \text{otherwise.}
	\end{cases}
\]
The global Cauchy index $\operatorname{ind}(f/g)$ is the sum of the local Cauchy indices.  The local Cauchy index of $f/g$ at $r$ equals the local topological of the \emph{reciprocal} function $g/f$ and similarly with the global degree.

\section{Connection with work of Eisenbud--Levine and Khimshiashvili} \label{Section: Connection}
Here we explain how Cazanave's description of  $\deg^{\bbA^1}(f/g)$ as the B\'{e}zout matrix is related to the signature formula of Eisenbud--Levine and Khimshiashvili, as well as results of Palamodov and the present authors.

The signature formula computes the local topological degree of the germ $h_0  \colon (\bbR^{n}, 0) \to (\bbR^{n}, 0)$ of a $C^{\infty}$-function as the signature of a bilinear form.  The formula applies when $h_0$ has the property that the local algebra 
\begin{equation} \label{Eqn: DefOfGermAlg}
	Q_{0}(h_0)  := C^{\infty}_{0}(\bbR^{n})/(h_0).
\end{equation}
has finite length.  Here $C^{\infty}_{0}(\bbR)$ is the ring of germs of smooth real-valued functions on $\bbR^n$ based at $0$ and $(h_0)$ is the ideal generated by the components of $h_0$.  

The bilinear form appearing in the signature formula is defined in terms of the Jacobian element  $J := \det( \frac{\partial f_{i}}{\partial x_{j}} ) \in Q_{0}(h_0)$.  The formula states that if  $\eta \colon Q_{0}(h_0) \to \bbR$ is a $\bbR$-linear function satisfying $\eta(J) > 0$, then the symmetric bilinear pairing $\beta_{0}(\eta)$ on $Q_{0}(h_0)$ defined by 
\[
	\beta_{0}(\eta)( a, b) = \eta( a b)
\]
satisfies 
\begin{equation} \label{Eqn: Signature}
	\deg^{\bbR}_{0}(h_0) = \text{signature of $\beta_{0}(\eta)$.} 
\end{equation}
This is \cite[Theorem~1.2]{eisenbud77} and \cite{khimshiashvili}; see \cite{Khimshiashvili01} and \cite[Chapter~5]{arnold12}  for recent expositions. 

 Earlier Palamodov \cite[Corollary~4]{palamodov} proved a parallel result  that computes the topological degree $\deg^{\bbC}(h_0)$ of the complexification $h_0 \otimes \bbC \colon \bbC^{n} \to \bbC^{n}$ as
\begin{equation} \label{Eqn: Rank}
	\deg^{\bbC}(h_0) = \text{rank of $\beta_{0}(\eta)$} 
\end{equation}
under the assumption that $h_0$ is analytic (so a complexification exists).

In the special case where $n=1$ and $h_0$ is the germ of a rational function $f/g$, the signature formula is closely related to  Hurwitz's theorem.  For such a $h_0$, the natural map  $\bbR[x, 1/g] \to C^{\infty}_{0}(\bbR)$ induces an isomorphism
\[
	Q_{0}(f/g) \to Q_{0}(h_0)
\]
between $Q_{0}(h_0)$ and the localization $Q_{0}(f/g)$ of the algebra $Q(f/g) := \bbR[x, g^{-1}]/ (f g^{-1})$ at the maximal ideal $(x)$ of the origin, by the finiteness of $Q_{0}(h_0)$.

The algebra $Q(f/g)$ admits a distinguished functional $\operatorname{Res} \colon  Q(f/g) \to \bbR$ called the residue functional (or generalized trace) that is constructed using local duality.  The functional is related to the residues which are defined in complex analysis.  Extending scalars to $\bbC$, the residue functional satisfies
\[
	\operatorname{Res}(a) = \sum_{f(r)=0} \operatorname{Res}_{z=r}\left( g(z)/f(z) \cdot a(z)\right),
\]
where $\operatorname{Res}_{z=r}(p(z))$ is the residue of $p(z)$ at $z=r$ in the sense of complex analysis.

The matrix $S(f/g)$ is nothing other than a Gram matrix for $\beta(\operatorname{Res})$.  This should not be too surprising.  We essentially showed this in the proof of Theorem~\ref{Thm: NewCazanave}, where we computed the bilinear polynomial associated to  $S(f/g)$ as the coefficient of $1/x$ in $g(x)/f(x) \cdot \Theta(a; x) \Theta(b; x)$, i.e.~as the sum of the residues of $g(x)/f(x) \cdot \Theta(a; x) \Theta(b; x)$.   The identification of $S(f/g)$ as a Gram matrix is naturally derived from general results of Scheja--Storch, as we now explain.

Scheja--Storch described $\operatorname{Res}$ more generally over an arbitrary field using commutative algebra  in \cite{scheja}.  In the present context, the key result is their description of the bilinear pairing $\beta(\operatorname{Res})$ in terms of the element
\[
	\Delta := \frac{f/g(x) - f/g(y)}{x-y}.
\]

They show that if $v_1, \dots, v_\mu$ is a $k$-basis for $Q(f/g)$ and
\begin{gather}
	\Delta = \sum b_{i, j} v_{i}(x) v_{j}(y) \text{ modulo $f(x)/g(x)$ and $f(y)/g(y)$, then  }\\	
	\begin{pmatrix} b_{i, j} \end{pmatrix}  = \text{ the Gram matrix of $\beta$ with respect to the basis dual to $v_1, \dots, v_\mu$.} \label{Eqn: GramMatrix}
\end{gather}
This is not explicitly stated by Scheja--Storch, but it can be derived from  the formula on \cite[page~182, bottom]{scheja}.  That formula states that if $v_{1}^{\ast}, \dots, v_{\mu}^{*}$  is the basis dual to $v_1, \dots, v_{\mu}$ with respect to $\beta(\operatorname{Res})$, then $\Delta = v_1(x) v_{1}^{*}(y) + \dots + v_{\mu}(x) v_{\mu}^{*}(y)$.  The two bases are related by $v_{i} = \sum b_{i, j} v_{j}^{*}$, and the  equality \eqref{Eqn: GramMatrix} follows from substituting  $\sum b_{i, j} v_{j}^{*}(x)$ for $v_{i}(x)$ in $\sum v_i(x) v_i^{*}(y)$.  (Strictly speaking Scheja--Storch only work with $Q(f/g)$ when $g=1$ because they work with quotients of $k[x]$, but their argument applies to the localization $k[x, 1/g]$ with only notational changes.)

The matrices $\operatorname{B\acute{e}z}(f/g)$, $S(f/g)$, and $\operatorname{New}(f/g)$ are nothing other than Gram matrices for $\beta$ with respect to distinguished bases that we now define.  

\begin{df}\label{Definition: Bases}
	Define the \textbf{monomial basis} of $Q(f/g)$ to be $1/g, x/g, \dots, x^{\mu-1}/g$.
	
	Write $f(x) = x^{\mu} + a_{1} x^{\mu-1} + \dots a_{\mu-1} x+a_{\mu}$ and define the \textbf{Horner basis} of $Q(f/g)$ to be $( x^{\mu-1} + a_{1} x^{\mu-2} + \dots + a_{\mu-1} )/g$, $( x^{\mu-2} + a_{1} x^{\mu-3} + \dots + a_{\mu-2} )/g$, ..., $(x+a_{1})/g$, $1/g$.
	
	If $f$ factors into linear factors $f(x) = (x-r_{1})^{\mu_1} \dots (x-r_{m})^{\mu_m}$ over $k$, then the define the \textbf{Newton basis}  to be  $\frac{f(x)}{(x-r_{1}) g(x)}$, $\frac{f(x)}{(x-r_{1})^{2} g(x)}$, ..., $\frac{f(x)}{(x-r_{1})^{\mu_{1} -1} g(x)}$, $\frac{f(x)}{(x-r_{2}) g(x)}$, \dots, $\frac{f(x)}{(x-r_{m})^{\mu_{m}-1} g(x)}$.
\end{df}

Table~\ref{Table: GramMatrices}  identifies $\operatorname{B\acute{e}z}(f/g)$, $S(f/g)$, and $\operatorname{New}(f/g)$ as Gram matrices.  These identifications follow from our computations in Section~\ref{Section: Classical Proof}.  We get the identification of the  B\'{e}zout matrix  by dividing \eqref{Eqn: DefOfBezout} by $g(x) g(y)$.  If we instead divide  \eqref{Eqn: BezoutExpression}, we get the identification of $S(f/g)$.  Replacing the use of  $g/f = \sum s_{i}/x^i$ with  $g/f = \sum A_{r_i}(j)/(x-r_{i})^{j}$ in \eqref{Eqn: BezoutExpression} identifies $\operatorname{New}(f/g)$.

\begin{table}%[htdp]
\caption{Gram Matrices}
\begin{center}
\begin{tabular}{ l l}
Gram matrix						&		Dual basis \\
\hline \hline 
$\operatorname{B\acute{e}z}(f/g)$		&		Monomial basis  \\
$S(f/g)$ 							& 		Horner basis  \\
$\operatorname{New}(f/g)$				&		Newton basis \\
\end{tabular}
\end{center}
\label{Table: GramMatrices}
\end{table}

We point out that this identification of $\operatorname{B\acute{e}z}(f/g)$ as a Gram matrix is different from the identification that appears in \cite{cazanave}.  The relevant text is Lemma~C.2 of the preprint version (but not in the published version) of loc.~cit.~that is availible on the online repository the arXiv as arXiv:0912.2227v2.  There Cazanave identifies the inverse matrix $\operatorname{B\acute{e}z}(f/g)^{-1}$, and hence $\operatorname{B\acute{e}z}(f/g)$, as a Gram matrix in an important special case, namely the case where  $f(x)$ and $g(x)$ have indeterminant coefficients and $k$ is a suitable algebraic closure.  His discussion, however, applies more generally to the case where  $f(x)$ has only simple roots (so $Q(f/g)$ is \'{e}tale).   He shows that $\operatorname{B\acute{e}z}(f/g)^{-1}$ is a Gram matrix for the bilinear pairing  $\beta(\phi)$ on $Q(f/g)$ defined by functional 
\begin{equation}	\label{Eqn: CazaFunctional}
	\phi(a) = \operatorname{Tr}\left(\frac{a}{f'(x) g(x)}\right).
\end{equation}
Here $\operatorname{Tr} \colon Q(f/g) \to k$ is the algebra trace.  Observe that $f'(x)$ is a unit in $Q(f/g)$ because we assumed that $f(x)$ has distinct roots, so \eqref{Eqn: CazaFunctional} is well-defined.

The functional $\phi$ is not equal to $\operatorname{Res}$.  Indeed, when $f(x)$ has simple roots, Scheja--Storch show 
\[
	\operatorname{Res}(a) = \operatorname{Tr} \left( \frac{a g(x)}{f'(x)} \right)
\]
This is  \cite[(4.2)~Satz]{scheja} (specifically, it is equivalent to the equality $\operatorname{Sp} = J \cdot \eta$ in loc.~cit.).  The two functionals do, however, define isomorphic symmetric bilinear forms since multiplication by $1/g(x)$ transforms $\beta(\operatorname{Res})$ into $\beta(\phi)$.

Having explained the connection between the residue pairing $\beta(\operatorname{Res})$ and Hurwitz's theorem, we now explain the relation to the signature formula. Since $Q(f/g)$ is an artin algebra, $Q_{0}(f/g)$ is a direct summand of $Q(f/g)$.  Scheja--Storch show that the restriction $\operatorname{Res}_{0}$ of the residue functional satisfies $\operatorname{Res}_{0}(J) = \operatorname{length} Q_{0}(f/g)$, so in particular, the restriction $\beta_{0}(\operatorname{Res}_{0})$ of $\beta(\operatorname{Res})$ to $Q_{0}(f/g)$ satisfies the hypothesis of the signature formula.  We conclude that 
\[
	\deg^{\bbR}_{0}(f/g) = \text{ the signature of $\beta_0(\operatorname{Res}_{0})$.}
\]
In this manner, applied to a rational function $f/g$, the signature formula is the local analogue of Hurwitz's theorem.

Cazanave's theorem is an enrichment of Hurwitz's theorem, and it implies the following analogous enrichment of the signature formula:
\begin{equation} \label{Eqn: LocalCazanave}
	\deg^{\bbA^1}_{0}(f/g) = \beta_0(\operatorname{Res}_{0}) \text{ in $\operatorname{GW}(k)$.}
\end{equation}
(This is Corollary~\ref{Corollary: NewtonIsDegreeLocal}.)

Equality \eqref{Eqn: LocalCazanave} answers a question of Eisenbud for polynomial maps of $1$ variable.  Eisenbud, in a survey article on his work with Levine,   observed that $\beta_0(\operatorname{Res}_{0})$ is defined for a polynomial map $f \colon \bbA^{n}_{k} \to \bbA^{n}_{k}$ with an isolated zero at the origin when $k$ is an arbitrary field of odd characteristic.  He then  proposed $\beta_0$ as the definition of the degree and asked if this degree has an interpretation or usefulness, say in cohomology theory \cite[Some remaining questions (3)]{eisenbud78}.  

In this paper we have shown that, in odd characteristic, the stable isomorphism class of $\beta_0(\operatorname{Res}_{0})$ is the local degree in $\bbA^1$-homotopy theory when $f$ is a polynomial map of $1$ variable.  In the companion paper \cite{wickelgren16}, the present authors  answer Eisenbud's question in full generality: the main theorem of that paper (Theorem~\ref{Thm: CompanionMainTheorem}) states that \eqref{Eqn: LocalCazanave} remains valid when $f/g$ is replaced by a polynomial map $f \colon \bbA^{n}_{k} \to \bbA^{n}_{k}$ in any number of variables that has  an isolated zero at the origin.

\section{Acknowledgements}
TO BE ADDED AFTER THE REFEREEING PROCESS. %thank candace for proof-reading.

Jesse Leo Kass was partially sponsored by the National Security Agency under Grant Number H98230-15-1-0264 and  the Simons Foundation under Award Number 429929.  The United States Government is authorized to reproduce and distribute reprints notwithstanding any copyright notation herein. This manuscript is submitted for publication with the understanding that the United States Government is authorized to reproduce and distribute reprints.

Kirsten Wickelgren was partially supported by National Science Foundation Awards DMS-1406380 and DMS-1552730.

\providecommand{\bysame}{\leavevmode\hbox to3em{\hrulefill}\thinspace}
\providecommand{\MR}{\relax\ifhmode\unskip\space\fi MR }

\providecommand{\MRhref}[2]{%
  \href{http://www.ams.org/mathscinet-getitem?mr=#1}{#2}
}
\providecommand{\href}[2]{#2}

\end{document}